\documentclass{amsart}
\usepackage{amsmath,amssymb,amscd,latexsym}
\usepackage{float}

\usepackage[all]{xy}

\usepackage{graphicx}

\usepackage{mathrsfs}

\usepackage{stmaryrd}

\usepackage{enumitem}
\usepackage{nicefrac}
\usepackage{multirow}

\usepackage{tikz}
\usepackage{tikz-cd}

\usepackage{hyperref}

\usepackage[normalem]{ulem}

\definecolor{NTNUblue}{RGB}{0,80,158}
\definecolor{NTNUbluesupport}{RGB}{62,98,138}
\definecolor{NTNUorange}{RGB}{239,129,20}

\newcommand{\Z}{{\mathbb{Z}}}

\newcommand{\ev}{\mathrm{ev}}

\newcommand{\modu}[3]{#1 \equiv #2 \; (\mathrm{mod} \; #3)}

\newcommand{\Hfree}[2]{H_\text{free}^{#1}(#2;\mathbb{Z})}
\newcommand{\Imm}{\mathrm{Im}\,}

\newcommand{\Map}{\mathrm{Map}}

\newcommand{\mapbso}{\mathrm{Map}_f(S^2,BSO(6))}
\newcommand{\mapbe}{\mathrm{Map}_f(S^4,BE_7)}
\newcommand{\BGa}{B\mathcal{G}(\xi)}
\newcommand{\Ga}{\mathcal{G}}

\newcommand{\Spin}{\mathrm{Spin}}

\newcommand{\Sq}{\mathrm{Sq}}

\newtheorem{theorem}{Theorem}[section]
\newtheorem{lemma}[theorem]{Lemma}

\newtheorem*{theorem*}{Theorem}

\theoremstyle{definition}

\newtheorem{remark}[theorem]{Remark}

\usepackage{adjustbox}

\begin{document}

\author{Eiolf Kaspersen}
\address{Department of Mathematical Sciences, NTNU, NO-7491 Trondheim, Norway}
\email{eiolf.kaspersen@ntnu.no}

\author{Gereon Quick}
\thanks{The second-named author was partially supported by the RCN Project No.\,313472.}
\address{Department of Mathematical Sciences, NTNU, NO-7491 Trondheim, Norway}
\email{gereon.quick@ntnu.no}

\title[On the Thom morphism for certain Lie and gauge groups]{A note on the Thom morphism for the classifying space of certain Lie groups and gauge groups}

\date{}

\begin{abstract}
We give a complete description of which non-torsion generators are not in the image of the Thom morphism from complex cobordism to integral cohomology for the classifying space of exceptional Lie groups except for $E_8$. 
We then show that the Thom morphism is not surjective for the classifying space of the gauge group of a principal $E_7$-bundle over the four-dimensional sphere. 
We use the results to detect nontrivial elements in the kernel of the reduced Thom morphism for Lie groups and their classifying spaces.  
\end{abstract}

\maketitle

\section{Introduction}

Let $G$ be a compact, connected Lie group, and let $\xi$ denote a principal $G$-bundle 
\begin{equation*}
\begin{tikzcd}
    G \arrow[r, hookrightarrow] & P \arrow[r, "\pi"] & X
\end{tikzcd}
\end{equation*}
over a paracompact space $X$.
Recall that the \textit{gauge group} $\Ga(\xi)$ of $\xi$ is defined to be the group of automorphisms of $\xi$, i.e., 
\begin{equation*}
\Ga(\xi) = \{\phi \in \mathrm{Aut}_G(P) \;\vert\;\pi \circ \phi = \pi\}.
\end{equation*}
Gauge groups play an important role in geometry, topology and mathematical physics.   
Moreover, the classifying space $B\Ga(\xi)$ is homotopy equivalent to the moduli space of connections on $\xi$. 
Integral singular cohomology $H^*(B\Ga(\xi);\Z)$ is a fundamental invariant of $B\Ga(\xi)$, and it is an important question which elements in $H^*(B\Ga(\xi);\Z)$ are the fundamental class of a smooth manifold $M \to B\Ga(\xi)$, i.e., for which classes $[c]$ in $H^*(B\Ga(\xi);\Z)$ is there an oriented, compact, smooth manifold $M$ and a continuous map $f \colon M \to B\Ga(\xi)$ such that $f_*([M]) = [c]$ where $[M]$ denotes the Poincar\'e dual of the fundamental class of $M$ in $H_*(M;\Z)$. 
This question is closely related to the Thom morphism 
\[
\tau \colon MU^*(B\Ga(\xi)) \longrightarrow H^*(B\Ga(\xi);\Z)  
\]
from complex cobordism to singular cohomology 
since complex cobordism classes over $B\Ga(\xi)$ can be represented by complex-oriented proper maps $f \colon M \to B\Ga(\xi)$, where $M$ is a smooth manifold, and $\tau$ sends the cobordism class $[f]$ to the cohomology class $f_*[M]$. 
In fact, the Thom morphism plays a key role in algebraic and geometric topology.  \\

The purpose of the present paper is to show that the Thom morphism for the classifying space of certain gauge groups is not surjective and thereby to show that there is a restriction for how non-torsion classes can be represented by fundamental classes. \\

The cohomology of $B\mathcal{G}(\xi)$ is closely related to that of $BG$, the classifying space of the Lie group $G$.  
The torsion in the cohomology of $\BGa$ has been studied  in \cite{kameko23}, \cite{minowa} and \cite{tsukuda}, see also \cite{KT}.  
When the cohomology has torsion, the Thom morphism can be non-surjective. 
The question of when the Thom morphism is surjective for $BG$ has previously been studied in for example \cite{totaro} and \cite{bo}  
where the non-surjectivity of the Thom morphism for classifying spaces of certain Lie groups is used to detect new phenomena of the cycle map and Deligne cohomology in algebraic geometry. 

In Section \ref{sec:BG}, we determine the image of the Thom morphism for the classifying spaces of $SO(2n)$ 
as well as the exceptional Lie groups, with the exception of $E_8$ where we only provide a partial result. 
We note that partial results were already known for $SO(4)$ and $G_2$.  
Therefore, our results significantly extend the known cases by providing a complete list of which non-torsion generators are in the image for $G_2$, $F_4$, $E_6$ and $E_7$.     
It turns out that the classifying space of $E_7$ behaves in a slightly different way compared to the other exceptional Lie groups. 
We do not know of a geometric explanation of this phenomenon. 
We note that some of our results could also have been deduced from the computations on $BP$-cohomology of Kono and Yagita in \cite[Theorem 5.5]{ky}. 
In Section \ref{sec:gauge_groups} we study the image of the Thom morphism for the classifying spaces of the gauge groups of principal $SO(6)$- and $E_7$-bundles over spheres.  
Our main result is the following:

\begin{theorem}
Let $\xi$ be a principal $E_7$-bundle over $S^4$. 
Then there is a non-torsion generator in $H^4(\BGa;\Z)$ which is not in the image of the Thom morphism.
\end{theorem}

As explained in \cite{totaro}, the non-surjectivity of $\tau$ can be used to construct examples where the reduced Thom morphism
\begin{equation*}
\bar{\tau} \colon MU^*(X \times B\Z/p) \otimes_{MU^\ast} \Z \longrightarrow H^*(X \times B\Z/p;\Z)
\end{equation*}
is not injective. 
The non-injectivity of the reduced Thom morphism is crucial for the applications in algebraic geometry in \cite{totaro}. 
In Section \ref{sec:injectivity}, we show that $\bar{\tau}$ is not injective when $X$ is the classifying space of $SO(n)$ with $n$ even or an exceptional Lie group and $p=2$. 
We note that the work of Kono and Yagita in \cite{ky} implies a stronger statement on the integral reduced Thom morphism for the classifying spaces of $F_4$ and $E_6$.  
In Theorem \ref{thm:LG_not_injective}, we provide a complete list of the simplest cases of non-injectivity of $\bar{\tau}$ for $X$ a compact connected Lie group with simple Lie algebra. 


\subsection*{Acknowledgement}
We are grateful to Nobuaki Yagita and the anonymous referee for helpful comments and suggestions.


\section{Classifying spaces of Lie groups}\label{sec:BG}

Let $X$ denote a finite $CW$-complex. 
It is a well-known fact that all Steenrod operations of odd degree vanish on the image of $MU^*(X)$ in $H^*(X;\Z/p)$ for all prime numbers $p$ (see for example \cite[page 468]{totaro}, \cite[Proposition 3.6]{bo},  \cite{cartanSem}). 
Hence, in order to show that an element $x \in H^*(X;\Z)$ is not in the image of the Thom morphism, it suffices to find a Steenrod operation of odd degree which does not vanish on $r(x)$ where $r \colon H^*(X;\Z) \rightarrow H^*(X;\Z/p)$ denotes the reduction map. 
We will now apply this observation to the cases where $X$ is $BSO(n)$ or the classifying space of an exceptional Lie group. 
%
%
We let $\Hfree{\ast}{X}$ denote the quotient of $H^\ast(X;\Z)$ by the torsion subgroup. 
When there is no risk of confusion we often use the same notation for an element in $H^\ast(X;\Z)$ and its image in $\Hfree{\ast}{X}$.

\subsection{Special Orthogonal Groups}

Recall from \cite[Theorem III.5.16]{mt} that the free cohomology of $BSO(n)$ is given by
\begin{align*}
    \Hfree{\ast}{BSO(n)} \cong \begin{cases}
        \Z[e_4,e_8,\ldots,e_{2n-2}],  &\text{$n$ odd} \\
        \Z[e_4,e_8,\ldots,e_{2n-4},\chi_n],  &\text{$n$ even.}
    \end{cases}
\end{align*}

\begin{theorem}\label{thm:bso}
Let $n \geq 4$ be even and let $z\in H^n(BSO(n);\Z)$ be a torsion class (or $0$). Then $(\chi_n + z) \in H^n(BSO(n);\Z)$ is not in the image of the Thom morphism.
\end{theorem}
\begin{proof}
From \cite[Theorems III.3.19 and III.5.12]{mt} we know that the mod $2$ cohomology of $BSO(n)$ is given by
\begin{equation*}
H^\ast(BSO(n);\Z/2) \cong \Z/2[y_2,y_3,\ldots,y_n],
\end{equation*}
with
\begin{equation}\label{bsosq}
\Sq^j(y_k) = \sum_{i=0}^j \binom{k-i-1}{j-i} y_{k+j-i} y_i.
\end{equation}
Let $r$ denote the mod $2$-reduction map $H^\ast(BSO(n);\Z) \rightarrow H^\ast(BSO(n);\Z/2)$. 
Then
\begin{align*}
    r(e_{4k}) &= y_{2k}^2 + \text{(other terms)}\\
    r(\chi_n) &= y_n + \text{(other terms)},
\end{align*}
since $y_{2k}^2$ and $y_n$ are neither the image nor the source of a nontrivial Bockstein homomorphism. Furthermore, we have
\begin{equation*}
    \Sq^3(y_n) = \sum_{i=0}^3 \binom{n-1-i}{3-i} y_{n+3-i} y_i = y_3 y_n \neq 0.
\end{equation*}
We will now show that no other term in $r(\chi_n + z)$ is mapped to $y_3y_n$ by $\Sq^3$. 
It follows from equation \eqref{bsosq} that the only other element of $H^{n-3}(BSO(n);\Z/2)$ which can map to $y_3y_n$ under $\Sq^3$ is $y_2y_{n-2}$. 
For this element, the relevant part of the Bockstein diagram for $BSO(n)$ is the following:
\begin{equation*}
\begin{tikzcd}
 \text{Degree:} & n & n+1 & n+2 \\
 \text{Generators: } & y_2y_{n-2} \arrow[r] \arrow[dr]  & y_3 y_{n-2} \arrow[dr] \\
 & & y_2 y_{n-1} \arrow[r] & y_3 y_{n-1}
 \end{tikzcd}
 \end{equation*}
 where there are no other nontrivial Bocksteins going into or out of any of these elements (see for example \cite[Chapter 3E]{hatcher} and \cite[Section 2.2]{KQ} for the use of Bockstein cohomology). 
 It follows that $y_2y_{n-2}$ does not generate a nontrivial element of the Bockstein cohomology of $BSO(n)$, and thus it is not one of the summands in $r(\chi_n + z)$. 
 Thus, we have $\Sq^3(r(\chi_n+ z)) \neq 0$. 
 The statement then follows from the fact that all odd-degree elements of the Steenrod algebra vanish on the image of the Thom morphism.
\end{proof}

\subsection{Exceptional Groups}

Recall from \cite[Theorems VI.5.5 and VI.5.10]{mt} that the free cohomologies of the classifying spaces of the simply connected exceptional Lie groups are given by
\begin{align*}
    \Hfree{\ast}{BG_2} &\cong \Z[e_4,e_{12}]\\
    \Hfree{\ast}{BF_4} &\cong \Z[e_4,e_{12},e_{16},e_{24}]\nonumber\\
    \Hfree{\ast}{BE_6} &\cong \Z[e_4,e_{10},e_{12},e_{16},e_{18},e_{24}]\nonumber\\
    \Hfree{\ast}{BE_7} &\cong \Z[e_4,e_{12},e_{16},e_{20},e_{24},e_{28},e_{36}]\nonumber\\
    \Hfree{\ast}{BE_8} &\cong \Z[e_4,e_{16},e_{24},e_{28},e_{36},e_{40},e_{48},e_{60}].\nonumber
\end{align*}

\begin{theorem}\label{thm:thom_BG}
Let $G=$ $G_2$, $F_4$ or $E_6$. Then the generator $e_4\in H^\ast(BG;\Z)$ is not in the image of the Thom morphism, while all other non-torsion generators are in the image.
\end{theorem}
\begin{proof}
The mod $2$ cohomology of each of the classifying spaces is given by
    \begin{align*}
        H^\ast(BG_2;\Z/2) &\cong \Z/2[y_4,y_6,y_7,y_{10}] \\
        H^\ast(BF_4;\Z/2) &\cong \Z/2[y_4,y_6,y_7,y_{16},y_{24}] \nonumber \\
        H^\ast(BE_6;\Z/2) &\cong \Z/2[y_4,y_6,y_7,y_{10},y_{18},y_{32},y_{34},y_{48}] / I, \nonumber
    \end{align*}
    where $I$ denotes the ideal given by 
    \begin{equation*}
        I = \langle y_7y_{10},y_7y_{18},y_7y_{34},y_{34}^2 + y_{18}^2y_{32} + y_{10}^2y_{48} + y_6y_{10}y_{18}y_{34} + y_4y_{10}y_{18}^3 + y_4y_{10}^3y_{34}\rangle
    \end{equation*}
see \cite[Corollary VII.6.3 and Theorem VII.6.6]{mt} and \cite[Theorem 1.1 and Proposition 5.1]{knn}. 
In each of the four cases, 
the reduction map $H^\ast(BG;\Z) \rightarrow H^\ast(BG;\Z/2)$ sends $e_4$ to $y_4$, and in each case $\Sq^3(y_4) = y_7$. 
This proves the first claim since $\Sq^3$ vanishes on the image of the Thom morphism as we pointed out at the beginning of Section \ref{sec:BG}. 
For the other generators, the claim follows from the fact that all differentials in the Atiyah--Hirzebruch spectral sequence vanish as explained, for example, in \cite[Section 2.1]{KQ} or \cite[page 471]{totaro}. 
To prove the latter assertion, we note first that all differentials in the Atiyah--Hirzebruch spectral sequence for the respective spaces are torsion, i.e., their images are contained in the subgroup of torsion elements. 
For $BG_2$, the cohomology has only $2$-torsion. 
Therefore, it suffices to show that all odd degree compositions of Steenrod squares vanish. 
This can be checked directly for $e_{12}$.  
For $BF_4$ and $BE_6$, 
the respective cohomology groups contain both $2$- and $3$-torsion. 
The cohomology operations in $H^\ast(BF_4;\Z/3)$ and $H^\ast(BE_6;\Z/3)$ are all power operations for $p=3$.  
The differentials in the Atiyah--Hirzebruch spectral sequence are 
compositions of Steenrod squares and power operations. 
We can then check directly that all such compositions which increase the cohomological degree by an odd number vanish on the respective generators. 
This proves the second claim and finishes the proof of the theorem. 
\end{proof}

As the following theorem shows, the situation is a bit different for the classifying space $BE_7$. 
While in the previous cases only the generator in degree $4$ is not hit by the Thom morphism, 
for $E_7$ there are several generators which are not in the image of $\tau$. 

\begin{theorem}\label{thm:thom_BE7}
The generators $e_4,e_{16},e_{24},e_{28} \in H^\ast(BE_7;\Z)$ are not in the image of the Thom morphism, and nor are the sums of any of these generators with a $2$-torsion element in the same degree. 
The non-torsion generators $e_{12}$ and $e_{20}$ are in the image of the Thom morphism. 
For the generator $e_{36}\in \Hfree{36}{BE_7}$, there exist lifts to $H^{36}(BE_7;\Z)$ which are not in the image of the Thom morphism, while other lifts are in the image. 
\end{theorem}
\begin{proof}
    The mod $2$ cohomology of $BE_7$ is given by
    \begin{align*}
        H^\ast(BE_7;\Z/2) &\cong \Z/2[y_4,y_6,y_7,y_{10},y_{11},y_{18},y_{19},\nonumber\\
        &\qquad\quad\; y_{34},y_{35},y_{64},y_{66},y_{67},y_{96},y_{112}] / J, \nonumber
    \end{align*}
    where $J$ denotes the ideal given by 
    \begin{align*}
        J = &\langle 
        y_6y_{11} + y_7y_{10}, ~
        y_6y_{19} + y_7y_{18}, ~
        y_{10}y_{19} + y_{11}y_{18}, ~
        y_{11}^3 + y_7^2y_{19}, \\
        & y_6y_{35} + y_7y_{34}, ~
        y_{10}y_{35} + y_{11}y_{34}, ~
        y_{11}y_{19}^2,  ~
        y_{18}y_{35} + y_{19}y_{34}, \\
       & y_{19}^3, ~
        y_6y_{67} + y_7y_{66}, ~
        y_{10}y_{67} + y_{11}y_{66}, ~
        y_{18}y_{67} + y_{19}y_{66}, \\
        & y_7^2y_{67} + y_{11}y_{35}^2, ~
        y_{11}^2y_{67} + y_{19}y_{35}^2, ~ 
        y_{34}y_{67} + y_{35}y_{66}, ~
        y_{19}^2y_{67}, \\
       & y_{66}^2 + y_{10}^2y_{112} + y_{18}^2y_{96}, ~
        y_{66}y_{67} + y_{10}y_{11}y_{112} + y_{18}y_{19}y_{96}, \\
        & y_{67}^2 + y_{11}^2y_{112} + y_{19}^2y_{96}, ~
        y_{35}^2y_{67} + y_{7}^2y_{11}y_{112} + y_{11}^2y_{19}y_{96}, \\
        & y_{34}^2 + y_6^4y_{112} + y_{10}^4y_{96} + y_{18}^4y_{64}, \\
        & y_{34}^3y_{35} + y_6^3y_7y_{112} + y_{10}^3y_{11}y_{96} + y_{18}^3y_{19}y_{64}, \\
        & y_{34}^2y_{35}^2 + y_6^2y_7^2y_{112} + y_{10}^2y_{11}^2y_{96} + y_{18}^2y_{19}^2y_{64}, \\
        & y_{34}y_{35}^3 + y_6y_7^3y_{112} + y_7^2y_{11}y_{18}y_{96}, ~
        y_{35}^4 + y_7^4y_{112} + y_{11}^4y_{96} 
        \rangle,
    \end{align*}
see \cite[Theorem 4.4]{mimura}. 
From \cite[page 276 and Corollary 6.9]{kms} we then get that the Steenrod algebra acts on the following generators as  
    \begin{align*}
        \Sq^3(r(e_4)) &= \Sq^3(y_4) = y_7 \\
        \Sq^{15}(r(e_{16})) &= \Sq^{15}(y_6 y_{10}) = y_6^2 y_{19} + y_{10}^2 y_{11} + y_4 y_7 y_{10}^2 \\
        \Sq^3(r(e_{24})) &= \Sq^3(y_6 y_{18}) = y_7 y_{10}^2. 
    \end{align*}
This implies, in particular, that these elements are not in the image of the Thom morphism, since all Steenrod operations of odd degree vanish on the image of the Thom morphism (see the beginning of Section \ref{sec:BG} for references for this claim). 
By analysing the Bockstein homomorphisms, we see that an element of $H^{28}(BE_7;\Z)$ which maps to the generator $e_{28} \in \Hfree{28}{BE_7}$ can be mapped to either $y_{10}y_{18}$ or $y_{10}y_{18} + y_7^4$ by $r$. 
Both of these elements are mapped to $y_{10}^2 y_{11}$ by $\Sq^3$, so $e_{28}$ plus torsion is not in the image of the Thom morphism either. 

For the generators $e_{12}$ and $e_{20}$, the claim follows again from the fact that all differentials in the Atiyah--Hirzebruch spectral sequence vanish as explained, for example, in \cite[Section 2.1]{KQ} or \cite[page 471]{totaro}.  
To prove the latter assertion, we use again that all differentials in the Atiyah--Hirzebruch spectral sequence for the respective spaces are torsion. 
The cohomology groups of $BE_7$ contain both $2$- and $3$-torsion. 
The cohomology operations in $H^\ast(BE_7;\Z/3)$ are all power operations for $p=3$.  
The differentials in the Atiyah--Hirzebruch spectral sequence are thus 
compositions of Steenrod squares and power operations. 
We can then check directly that all such odd degree compositions vanish on the images of the generators $e_{12}$ and $e_{20}$. 

Finally, we have the following Bockstein diagram in degree $36$: 
    \begin{equation*}
\adjustbox{scale=0.8}{
\begin{tikzcd}[row sep=tiny, cramped]
35 & 36 & 37 \\
y_4 y_7^3 y_{10} \arrow[to=47311] & |[alias=47311]| y_4 y_7^3 y_{11} & y_{18} y_{19} \\
y_7^2 y_{10} y_{11} \arrow[to=72112] & |[alias=72112]| y_7^2 y_{11}^2 & |[alias=47219]| y_4 y_7^2 y_{19} \\
y_4^2 y_6 y_7^3 \arrow[to=4274] & |[alias=4274]| y_4^2 y_7^4 & |[alias=421118]| y_4^2 y_{11} y_{18} \\
y_{35}  & |[alias=71118]| y_7 y_{11} y_{18} \arrow[to=71119] & |[alias=71119]| y_7 y_{11} y_{19} \\
y_4^4 y_{19} & y_4 y_{10} y_{11}^2 \arrow[to=4113] & |[alias=4113]| y_4 y_{11}^3 \\
y_7 y_{10} y_{18} & y_4 y_6^3 y_7^2 \arrow[to=46273] & |[alias=46273]| y_4 y_6^2 y_7^3 \\
y_4 y_6 y_7 y_{18} & y_4^4 y_6 y_7^2 \arrow[to=4473] & |[alias=4473]| y_4^4 y_7^3 \\
y_4 y_{10}^2 y_{11} & y_{18}^2 & |[alias=62718]| y_6^2 y_7 y_{18} \\
y_4^6 y_{11} & y_4^2 y_{10} y_{18} & |[alias=43718]| y_4^3 y_7 y_{18} \\
y_4^2 y_7 y_{10}^2 & y_6^3 y_{18} & |[alias=427112]| y_4^2 y_7 y_{11}^2 \\
y_6^3 y_7 y_{10} & y_4^3 y_6 y_{18}  & y_4^4 y_{10} y_{11} \\
y_4^3 y_6 y_7 y_{10}  & y_6 y_{10}^3  & |[alias=437211]| y_4^3 y_7^2 y_{11} \\
y_4 y_6^4 y_7 & y_4^4 y_{10}^2  & |[alias=7103]| y_7 y_{10}^3 \\
y_4^4 y_6^2 y_7 & y_4 y_6^2 y_{10}^2 & y_4 y_6 y_7 y_{10}^2 \\
y_4^7 y_7 & y_4^2 y_6^3 y_{10} & y_6 y_7^3 y_{10} \\
     & y_4^5 y_6 y_{10}  & |[alias=4262710]| y_4^2 y_6^2 y_7 y_{10}\\
     & y_6^6 & |[alias=45710]| y_4^5 y_7 y_{10}  \\
     & y_4^3 y_6^4 & y_4^3 y_6^3 y_7 \\
    & y_4^6 y_6^2 & y_6^5 y_7 \\
 & y_4^9 & y_4^6 y_6 y_7 
\end{tikzcd}}
\end{equation*}
The diagram shows that an element of $H^{36}(BE_7;\Z)$ corresponding to the generator $e_{36}\in \Hfree{36}{BE_7}$ is mapped to $y_{18}^2 + L$ by $r$, where $L$ is some linear combination of the elements
\begin{equation*}
    y_4y_7^3y_{11},\; y_7^2 y_{11}^2,\; y_4^2 y_7^4.
\end{equation*}
While all odd-degree elements of the Steenrod algebra act trivially on $y_{18}^2$, $y_7^2 y_{11}^2$ and $y_4^2 y_7^4$, we have $\Sq^3(y_4y_7^3y_{11}) = y_7^4 y_{11}$. 
Thus, any lift of $e_{36}$ which contains the term $y_4y_7^3y_{11}$ is not in the image of the Thom morphism, while any lift which does not contain that term is in the image.
This proves the last claim and finishes the proof of the theorem. 
\end{proof}

\begin{remark}
As for the other exceptional Lie groups, the generator $e_4\in H^4(BE_8;\Z)$ is not in the image of the Thom morphism. 
This can for example be shown using the fact that $BE_8$ and the Eilenberg--MacLane space $K(\Z,4)$ have homotopy equivalent $15$-skeletons \cite[page 185]{hill}. 
However, since the mod $2$ cohomology of $BE_8$ is not known, we cannot give a complete answer to which other generators are in the image of the Thom morphism.
\end{remark}

\begin{remark}\label{rem:KY_alternative}
We note that some of our results could have been deduced from the computations of $BP$-cohomology of Kono and Yagita. 
Moreover, it follows from \cite[Theorem 5.5]{ky} that $e_4$ in $H^4(BG;\Z)$ is not in the image of $\tau$ for $G=SO(4)$ whereas $2e_4$ is in the image. 
\end{remark}


\section{Gauge Groups}\label{sec:gauge_groups}

Let $X$ be a paracompact space, and let $\xi$ be a principal $G$-bundle over $X$. 
The bundle $\xi$ is classified by a map $f\colon X \rightarrow BG$. 
Let $\Map_f(X,BG)$ denote the path component of $\Map(X,BG)$ which contains the map $f$. 
By \cite[Proposition 2.4]{ab}, $\BGa$ equals  $\Map_f(X,BG)$ in homotopy theory, and 
we will consider $\Map_f(X,BG)$ as a model for $\BGa$.  
For $X=S^n$, there is a fibre sequence
\begin{equation}\label{BGfibseq}
\begin{tikzcd}
    \Omega_f^n(BG) \arrow[r, hookrightarrow] & \mathrm{Map}_f(S^n,BG) \arrow[r, "\mathrm{ev}"] & BG,
\end{tikzcd}
\end{equation}
where $\Omega_f^n(BG)$ denotes the path component of $\Omega^n(BG)$ which contains $f$ and $\mathrm{ev}$ denotes the evaluation map at the basepoint of $S^n$. 
Note that maps in $\Map(S^n,BG)$ are not required to be pointed, while maps in $\Omega^n(BG)$ are. 
We further note that $\Omega BG \cong G$ and that all path-components of $\Omega^{n-1}(G)$ are homotopy equivalent. 
Therefore, sequence (\ref{BGfibseq}) simplifies to
\begin{equation}\label{BGfibseq2}
\begin{tikzcd}
    \Omega_0^{n-1}(G) \arrow[r, hookrightarrow] & \Map_f(S^n,BG) \arrow[r, "\mathrm{ev}"] & BG,
\end{tikzcd}
\end{equation}
where $\Omega_0^{n-1}(G)$ denotes the path-component of $\Omega^{n-1}(G)$ which contains the constant map.
%
Since there is a cross-section $s \colon BG \to \Map_f(S^n,BG)$ such that $\mathrm{ev}\circ s$ is the identity, 
the Thom morphism is easily seen to be non-surjective for $\Map_f(S^n,BG)$ if it is non-surjective for $BG$. 
When the Thom morphism is not surjective for $BG$, however, 
then the question of surjectivity of the Thom morphism for $\Map_f(S^n,BG)$ is more interesting. 
First we need the following lemma. 

\begin{lemma}\label{lem:low_degree_cohomology}
The cohomology group $\widetilde{H}^k(\Omega_0^3(E_7);\Z/2)$ is trivial for $k \leq 7$.
\end{lemma}
\begin{proof}
Let $Q_i$ denote the $i$th Dyer--Lashof operation \cite[Definition 4.1]{ka}
\begin{equation*}
    Q_i\colon H_k(\Omega^n X ;\Z/2) \longrightarrow H_{2k+i}(\Omega^n X;\Z/2)
\end{equation*}
for $0 \leq i \leq n-1$, and 
let $\beta \colon H_k(X,\Z/2) \to H_{k-1}(X;\Z/2)$ denote the Bockstein homomorphism.
By \cite[Theorem 3.15]{cy} there is an isomorphism 
    \begin{align*}
        H_\ast(\Omega_0^3(E_7);\Z/2) \cong \Z/2[Q_1^a\beta (u_{30})] \otimes \Z/2[Q_1^a Q_2^b (u_{30})] \otimes \left(
        \bigotimes_{k\in J} H_\ast(\Omega^3(S^k);\Z/2)\right),
    \end{align*}
where $J = \{11,15,19,23,27,35\}$ and where $a$ and $b$ range over all non-negative integers. 
Since the reduced homology of the triple loop space of a sphere of dimension at least $11$ is concentrated in degrees $\geq 8$ (see \cite[page 74]{choi96}), we get that $\widetilde{H}_k(\Omega_0^3(E_7);\Z/2)$ is trivial for $k \leq 7$. By the universal coefficient theorem the same holds for $\widetilde{H}^k(\Omega_0^3(E_7);\Z/2)$.
\end{proof}

We now prove our main result.

\begin{theorem}\label{thm:e7_bundle}
Let $\xi$ be a principal $E_7$-bundle over $S^4$. 
Then the image of $\ev^\ast(e_4)$ in $H^4(\BGa;\Z)$ is not in the image of the Thom morphism.
\end{theorem}
\begin{proof}
Due to the homotopy equivalence $\BGa \simeq \mapbe$, it suffices to show that $\ev^\ast(e_4) \in H^4(\mapbe;\Z)$ is not in the image of the Thom morphism. 
We consider the mod $2$ cohomology Serre spectral sequence of the fibre sequence
\begin{equation*}
\begin{tikzcd}
    \Omega_0^3(E_7) \arrow[r, hookrightarrow] & \mapbe \arrow[r,"\ev"] & BE_7.
\end{tikzcd}
\end{equation*}
The $E_2$-page is given by 
\begin{equation*}
    E_2^{p,q} = H^p(BE_7;\Z/2) \otimes H^q(\Omega_0^3(E_7);\Z/2).
\end{equation*}
By Lemma \ref{lem:low_degree_cohomology}, the element $y_7 \in H^7(BE_7;\Z/2) \cong E_2^{7,0}$ is not in the image of any nontrivial differentials for degree reasons. Thus, $\ev^\ast(y_7) \in H^7(\mapbe;\Z/2)$ is nonzero. Furthermore, the proof of Theorem \ref{thm:thom_BE7} shows that $\Sq^3(r(e_4)) = y_7$. The commutative diagram
\begin{equation*}
\begin{tikzcd}[row sep=large]
     H^4(BE_7;\Z) \arrow[r, "\ev^\ast"] \arrow[d, "r"'] & H^4(\mapbe;\Z) \arrow[d, "r"] \\
     H^4(BE_7;\Z/2) \arrow[r, "\ev^\ast"] \arrow[d, "\Sq^3"'] & H^4(\mapbe;\Z/2) \arrow[d, "\Sq^3"] \\
     H^7(BE_7;\Z/2) \arrow[r, "\ev^\ast"] & H^7(\mapbe;\Z/2)
\end{tikzcd}
\end{equation*}
then shows that $\Sq^3 \circ r$ acts non-trivially on $\ev^\ast(e_4)$, which completes the proof. 
\end{proof}

\begin{remark}
We can similarly examine whether the Thom morphism is surjective for principal $SO(n)$-bundles. Let $\xi$ denote a principal $SO(n)$-bundle over $S^2$ where the classifying map $f$ is not homotopic to the constant map. It was shown in \cite[Theorem 1.2]{minowa} that $\BGa$ has torsion if and only if $n \geq 5$. 
Since we found in Theorem \ref{thm:bso} that the Thom morphism is non-surjective for $BSO(n)$ when $n$ is even, a natural candidate for a gauge group whose classifying space has a non-surjective Thom morphism is $\xi$ as above when $n=6$. 
Since $\Omega_f^2(BSO(6)) \cong \Omega_0(SO(6)) \cong \Omega(\Spin(6))$, there is a fiber sequence
    \begin{equation}\label{bso6_seq}
    \begin{tikzcd}
        \Omega(\Spin(6)) \arrow[r, hookrightarrow] & \mapbso \arrow[r, "\ev"] & BSO(6).
    \end{tikzcd}
    \end{equation}
By \cite[Lemma 2.2]{choi96}, we know that
    \begin{equation*}
        H^\ast(\Omega(\Spin(6));\Z/2) \cong \Z/2[a_2]/(a_2^4) \otimes \Gamma[c_6] \otimes \left( \bigotimes_{i=0}^\infty \Z/2[\gamma_{i}(b_4)]/(\gamma_{i}(b_4)^4) \right), 
    \end{equation*}
    where $\Gamma[x]$ denotes the divided power algebra generated by $\{\gamma_0(x),\gamma_1(x),\ldots\}$. 
It is then possible that the differential $d_9\colon H^8(\Omega(\Spin(6));\Z/2) \longrightarrow H^9(BSO(6);\Z/2)$ in the Serre spectral sequence of sequence \eqref{bso6_seq} maps $\gamma_1(b_4)$ to $y_3y_6$, which would imply that the element $y_3y_6$ does not survive to the $E_\infty$-page. 
However, we were unable to determine whether this differential is trivial or not and we therefore do not know whether the Thom morphism is surjective in this case.
\end{remark}


\section{Non-injectivity of the integral Thom morphism}\label{sec:injectivity}

Let $X$ be a CW-complex and let $X_k$ denote the $k$-skeleton. 
Recall that $MU^\ast(X)$ is said to \textit{satisfy the Mittag--Leffler condition} if for all $n \geq 0$ there exists some $m \geq n$ such that $\Imm (MU^\ast(X) \rightarrow MU^\ast(X_n)) = \Imm (MU^\ast(X_m) \rightarrow MU^\ast(X_n))$. 
It follows from the Milnor short exact sequence that if $MU^\ast(X)$ satisfies the Mittag--Leffler condition, then $MU^\ast(X) = \varprojlim MU^\ast(X_n)$.
Letting $MU^\ast$ act on $\Z$ (or $\Z/p$) by having all generators in nonzero degree act trivially on $\Z$, we get a tensor product $MU^\ast(X) \otimes_{MU^\ast} \Z$. This is isomorphic to $MU^\ast(X)$ modulo an ideal contained in the kernel of the Thom morphism. Thus, the reduced Thom morphism 
\[
MU^\ast(X) \otimes_{MU^\ast} \Z \longrightarrow H^\ast(X;\Z)
\]
is well-defined (see also \cite[page 470]{totaro}). 
The following lemma is a general version of \cite[Corollary 5.3]{totaro} which is formulated for classifying spaces of compact Lie groups only. 
The proof, however, is the same as in \cite{totaro}. 
For completeness, we provide the reader with the full argument here since we will apply the assertion later in Theorem \ref{thm:LG_not_injective}. 

\begin{lemma}\label{lem:noninjectivity}
    Let $X$ be a CW-complex of finite type such that $MU^\ast(X)$ satisfies the Mittag--Leffler condition. Let $p$ be a prime, $k$ an integer and $x \in H^k(X;\Z)$. Suppose that the image of the Thom morphism $MU^k(X) \rightarrow H^k(X;\Z)$ contains $px$ but no element $y$ such that $py = px$. Then the Thom morphism 
    \begin{equation*}
        MU^{k+2}(X\times B\Z/p)\otimes_{MU^\ast}\Z \rightarrow H^{k+2}(X\times B\Z/p;\Z)
    \end{equation*}
    is not injective.
\end{lemma}
\begin{proof}
We first show that the map $MU^k(X)\otimes_{MU^\ast} \Z/p \rightarrow H^k(X;\Z/p)$ is not injective. 
Suppose that $\alpha \in MU^k(X)$ maps to $px$ under the Thom morphism. 
If the element $(\alpha \otimes 1) \in MU^k(X) \otimes_{MU^\ast} \Z/p$ is zero, then $\alpha = p\beta$ for some $\beta \in MU^k(X)$. 
It follows that the Thom morphism maps $\beta$ to some $y$ with $py = px$, which contradicts the assumption. 
Therefore, $\alpha \otimes 1$ is nonzero. 
However, the Thom morphism maps $\alpha \otimes 1$ to $0 \in H^k(X;\Z/p)$. 
Thus $MU^k(X)\otimes_{MU^\ast} \Z/p \rightarrow H^k(X;\Z/p)$ is not injective.  
    
For all finite CW-complexes $Y$, there is an isomorphism
    \begin{equation*}
        MU^\ast(Y\times B\Z/p) \cong MU^\ast(Y)\otimes_{MU^\ast}MU^\ast(B\Z/p)
    \end{equation*}
    by \cite[Theorem 2']{landweber70}.
    This implies  
\begin{align*}
MU^\ast(X\times B\Z/p) & 
\cong \varprojlim MU^\ast(X_n \times B\Z/p) 
\cong \varprojlim \left( MU^\ast(X_n) \otimes_{MU^\ast} MU^\ast(B\Z/p) \right) \\
& \cong MU^\ast(X) \otimes_{MU^\ast} MU^\ast(B\Z/p).
\end{align*}
It then follows that
\begin{align*}
MU^\ast(X\times B\Z/p) \otimes_{MU^\ast} \Z 
& \cong \left(MU^\ast(X) \otimes_{MU^\ast}MU^\ast(B\Z/p)\right)  \otimes_{MU^\ast}\Z \\
& \cong \left( MU^\ast(X) \otimes_{MU^\ast} \Z \right) \otimes_{MU^\ast} \left( MU^\ast(B\Z/p) \otimes_\Z \Z \right) \\
& \cong \left( MU^\ast(X) \otimes_{MU^\ast} \Z \right) \otimes_{\Z} \left( MU^\ast(B\Z/p) \otimes_{MU^\ast} \Z \right) \\
& \cong \left( MU^\ast(X) \otimes_{MU^\ast} \Z \right) \otimes_{\Z} \left( \Z[c]/(pc) \right),
\end{align*}
where $\vert c \vert = 2$ and where the final isomorphism comes from the fact that the Thom morphism
\begin{equation*}
MU^\ast(B\Z/p) \otimes_{MU^\ast} \Z  \longrightarrow H^\ast(B\Z/p;\Z) \cong \Z[c]/(pc)
\end{equation*}
is an isomorphism. 
Moreover, we get 
\begin{align*}
& \left( MU^\ast(X) \otimes_{MU^\ast} \Z \right) \otimes_{\Z} \left( \Z[c]/(pc) \right) 
 \cong  \left(MU^\ast(X)\otimes_{MU^\ast}\Z\right)\![c]/(pc) \\
\cong & \left( MU^\ast(X)\otimes_{MU^\ast}\Z\right) \oplus \prod_{i=1}^\infty \left( MU^\ast(X) \otimes_{MU^\ast} \Z \right)c^i/(pc^i) \\
\cong & \left( MU^\ast(X)\otimes_{MU^\ast}\Z\right) \oplus \prod_{i=1}^\infty \left( MU^\ast(X)\otimes_{MU^\ast}\Z/p \right)c^i.
\end{align*}
Thus, the integral Thom morphism for the space $X \times B\Z/p$ is given by the composite map
    \begin{equation*}
    \begin{tikzcd}[cramped, column sep=tiny]
        \left( MU^\ast(X)\otimes_{MU^\ast}\Z\right) \oplus \prod_{i=1}^\infty \left( MU^\ast(X)\otimes_{MU^\ast}\Z/p \right)c^i \arrow[d] \\
        H^\ast(X;\Z) \oplus \prod_{i=1}^\infty H^\ast(X;\Z/p)c^i \arrow[r, "\cong"] & H^\ast(X;\Z) \otimes_\Z H^\ast(B\Z/p;\Z) \arrow[d, hookrightarrow] \\
        & H^\ast(X\times B\Z/p;\Z).
    \end{tikzcd}
    \end{equation*}
Since $MU^k(X)\otimes_{MU^\ast} \Z/p \rightarrow H^k(X;\Z/p)$ is not injective, we get that the subgroup $\left( MU^k(X)\otimes_{MU^\ast}\Z/p \right)\!c$ maps noninjectively to $H^{k+2}(X \times B\Z/p)$, and the statement follows.
\end{proof}

By combining the previous lemma with the results in Section \ref{sec:BG}, we get multiple examples of when the Thom morphism is not injective.

\begin{theorem}
    Let $G=$ $G_2$, $F_4$, $E_6$, $E_7$ or $SO(n)$ with $n \geq 4$ even. Then the Thom morphism 
    \begin{equation*}
        MU^6(BG \times B\Z/2) \otimes_{MU^\ast} \Z \longrightarrow H^6(BG \times B\Z/2;\Z)
    \end{equation*} is not injective. 
\end{theorem}
\begin{proof}
We first note that $H^4(BG;\Z) \cong \Z$, where $e_4$ denotes the canonical generator. By Theorems \ref{thm:bso}, \ref{thm:thom_BG} and \ref{thm:thom_BE7} $e_4$ is not in the image of the Thom morphism. However, by \cite[Theorem 1]{landweber72}, $MU^\ast(BG)$ satisfies the Mittag--Leffler condition, from which it follows that some integer multiple $ne_4$ is in the image of the Thom morphism. Assume that $n$ is minimal. 
Since $\Sq^3$ acts non-trivially on $e_4$, it follows that $n$ is a multiple of $2$. Setting $p=2$, we then see that the conditions of Lemma \ref{lem:noninjectivity} are satisfied, and the statement follows. In fact, assuming that $\alpha \in MU^4(BG)$ maps to $ne_4$, we see that the element 
    \begin{equation*}
        (\alpha \otimes 1)c \in (MU^4(BG) \otimes_{MU^\ast} \Z/2)c \subseteq MU^6(BG\times B\Z/2)
    \end{equation*}
    is mapped to $0$ in $H^k(BG\times B\Z/2;\Z)$.
\end{proof}

\begin{remark}
For $E_7$, we note that also other generators can be used to construct non-trivial elements in the kernel by Theorem \ref{thm:thom_BE7}. 
\end{remark}

\begin{remark}\label{rem:KY_stronger}
We note that, for some of the groups, the work of Kono and Yagita allows for a stronger statement. 
One can deduce from \cite[pages 795-796]{ky} that the reduced Thom morphism $MU^\ast(BG)\otimes_{MU^\ast}\Z \longrightarrow H^\ast(BG;\Z)$ is not injective for $G=F_4$ and $G=E_6$.
\end{remark}

If $X$ is a Lie group, and not the classifying space, the Thom morphism is in many cases not injective. 
In the following theorem we give a complete list for the simplest cases where injectivity fails.   

\begin{theorem}\label{thm:LG_not_injective}
Let $G$ be a compact connected Lie group with simple Lie algebra. 
For an integer $n\ge 1$, let $r$ denote the smallest natural number such that $2^r \mid n$.
The Thom morphism
    \begin{equation*}
        MU^k(G \times B\Z/p)\otimes_{MU^\ast} \Z \longrightarrow H^k(G \times B\Z/p;\Z)
    \end{equation*}
is not injective in the following cases:
\begin{table}[H]
\renewcommand{\arraystretch}{1.2}
\begin{tabular}{|l|l|l|l|}
\hline
\rm{Group}             & $n$ & $k$ & $p$ \\ \hline
$\Spin(n)$       & $n\geq7$ &  $5$   &  $2$   \\ \hline
$SO(n)$          &  $n\geq5$   &  $5$   &   $2$  \\ \hline
$Ss(n)$          &  $8 \mid n$   &  $5$   &   $2$  \\ \hline
$Ss(n)$          &  $8 \nmid n$   &  $9$   &   $2$  \\ \hline
$PSO(n)$           &  $8 \mid n$   &  $5$   &   $2$  \\ \hline
$PSO(n)$           &   $8 \nmid n,\; n\geq 10$  &  $9$   &   $2$  \\ \hline
$PSp(n)$         &  $n$ even   &  $2^{r+1} + 1$   &   $2$  \\ \hline
$SU(n)/\Gamma_l$ &   $4 \mid n$, $\modu{l}{2}{4}$  &  $2^r + 1$   &   $2$  \\ \hline
$G_2$            &     &  $5$   &   $2$  \\ \hline
$F_4$            &     &  $5$   &   $2$ or $3$  \\ \hline
$E_6$            &     &  $5$   &   $2$ or $3$ \\ \hline
$E_6/\Gamma_3$   &     &  $5$   &   $2$  \\ \hline
$E_7$            &     &  $5,17$   &   $2$  \\ \hline
$E_7/\Gamma_2$   &     &  $5$   &   $3$  \\ \hline
$E_7/\Gamma_2$   &     &  $17$   &   $2$  \\ \hline
$E_8$            &     & $5,17,25,29$    &   $2$  \\ \hline
\end{tabular}
\end{table}
\end{theorem}
\begin{proof}
This follows from \cite[Theorem 1.1]{KQ} and Lemma \ref{lem:noninjectivity}.
\end{proof}

\begin{remark}
If $MU^\ast(\BGa)$ satisfies the Mittag--Leffler condition for the bundle $\xi$ of Theorem \ref{thm:e7_bundle}, then the Thom morphism 
\[
MU^6(\BGa \times B\Z/2)\otimes_{MU^\ast} \Z \longrightarrow H^6(\BGa \times B\Z/2;\Z)
\]
is not injective. However, we were unable to determine whether this is the case.
\end{remark}


%
\bibliographystyle{amsalpha}

\end{document}